\documentclass[10pt]{amsart}
\usepackage{amssymb}
\usepackage{amscd}
\usepackage{amsmath}
\theoremstyle{plain}

\theoremstyle{definition}
\newtheorem{question}{Question}

\newtheorem{proposition}{Proposition}[section]
\newtheorem{theorem}[proposition]{Theorem}

\newtheorem{corollary}[proposition]{Corollary}
\newtheorem{lemma}[proposition]{Lemma}
\newtheorem{definition}[proposition]{Definition}
\theoremstyle{remark}

\DeclareMathOperator{\tcf}{tcf}

\DeclareMathOperator{\bd}{bd}

\newcommand{\pcfsig}{\text{\rm pcf}_{\sigma\text{\rm-com}}}
\newcommand{\pcftau}{\text{\rm pcf}_{\tau\text{\rm-com}}}
\newcommand{\ppgamma}{\text{\rm pp}_{\Gamma(\theta,\sigma)}}

\DeclareMathOperator{\cf}{cf}

\DeclareMathOperator{\PP}{PP}

\DeclareMathOperator{\pcf}{pcf}

\newcommand{\sk}{\vskip.05in}

\newcommand{\restr}{\upharpoonright}

\DeclareMathOperator{\cov}{cov}

\DeclareMathOperator{\reg}{{\sf Reg}}

\DeclareMathOperator{\pp}{pp}

\numberwithin{equation}{subsection}
\numberwithin{proposition}{subsection}
\begin{document}
\title[Pseudopower Dichotomy]{The Pseudopower Dichotomy}
\begin{abstract}
We investigate pseudopowers of singular cardinals and deduce some consequences for covering numbers at singular cardinals of uncountable cofinality.
\end{abstract}
\author{Todd Eisworth}
\email{eisworth@ohio.edu}
\maketitle
\section{Introduction}
\subsection{Overview}
We use pcf theory to establish some equalities between various pseudopowers at a singular cardinal, and use these to derive some ZFC conclusions in cardinal arithmetic.  Our proofs rest on a principle we call {\em the Pseudopower Dichotomy}. This generalization of Fact 1.9 on page 324 of~\cite{cardarith}) tells us that an arbitrary singular cardinal falls into one of two classes: either it is a strong limit in a very weak sense, or it is not.  The majority of the paper is concerned with obtaining results in pcf theory that will allow us to derive meaningful conclusions in both situations.

Our main application of the dichotomy  is a {\sf ZFC} result about the behavior of covering numbers at singular cardinals of uncountable cofinality. The result is quite general, but a typical special case (in which the parameters are chosen for amusement), tells us, for example, that if $\mu$ is a singular cardinal of cofinality $\aleph_6$, then
\begin{equation}
\label{amusement}
\cov(\mu,\mu,\aleph_9,\aleph_2)=\cov(\mu,\mu,\aleph_7,\aleph_2)+\cov(\mu,\mu,\aleph_9, \aleph_6).\footnote{Although the parameters were chosen for amusement, they are not chosen randomly.  We shall see that this particular equation needs $\cf(\mu)$ to be less than $\aleph_\omega$, and that the $\aleph_7$ and $\aleph_6$ appearing on the right-hand side are $\cf(\mu)^+$ and $\cf(\mu)$, respectively. }
\end{equation}
We will also use  recent work of Gitik~\cite{gitik} on the Shelah Weak Hypothesis to provide  complementary consistency results, showing us that in many such equations both terms on the right-hand side are required.

This equation~(\ref{amusement}) is unlikely to mean much to someone unfamiliar with Shelah's {\em Cardinal Arithmetic}~\cite{cardarith}, so our goal for the remainder of this introductory section is lead a reader with a basic knowledge of pcf theory to the point where the preceding paragraphs are understandable.

\subsection{Basic pcf theory} So what do we consider to be ``a basic knowledge of pcf theory''?  Certainly the material covered in Abraham and Magidor's chapter~\cite{AM} in the Handbook of Set Theory~\cite{handbook} is more than enough.  We use their work as a starting point, and any notation that we neglect to define comes from their presentation. The book~\cite{3germans} is an comprehensive account of the same material, and both Kojman's unpublished~\cite{abc} and the classic paper of Burke and Magidor~\cite{burkemagidor} provide ample coverage of the background material as well. The author's paper~\cite{ei} also addresses of some of the topics under consideration here.  Most of the pcf material we need has to do with properties of the ideals $J_{<\lambda}[A]$ and associated pcf generators. More precisely, we will need to use some results of Shelah showing that well-structured sets of generators can be found in certain circumstances.

\subsection{Covering numbers}
Our equation (\ref{amusement}) expresses an equality between covering numbers at a singular cardinal. These covering numbers are cardinal characteristics introduced by Shelah~\cite{cardarith} in his analysis of cardinal arithmetic, and they  arise naturally when one considers structures of the form $\left([\mu]^{<\kappa},\subseteq\right)$. Indeed, these covering numbers are a refinement of the idea of cofinality in such structures.
\begin{definition}
Suppose $\mu$, $\kappa$, $\theta$, and $\sigma$ are cardinals satisfying
\begin{equation}
\label{ordering}
2\leq\sigma<\theta\leq\kappa\leq\mu.
\end{equation}
A {\em $\sigma$-cover of $[\mu]^{<\theta}$ in $[\mu]^{<\kappa}$} is a family $\mathcal{P}\subseteq[\mu]^{<\kappa}$ with the property that every member of $[\mu]^{<\theta}$ can be covered by a union of fewer than $\sigma$ elements drawn from~$\mathcal{P}$, that is
\begin{equation}
(\forall X\in [\mu]^{<\theta})(\exists Y\subseteq\mathcal{P})\left[|Y|<\sigma\text{ and }X\subseteq\bigcup Y\right].
\end{equation}
The covering number $\cov(\mu,\kappa,\theta,\sigma)$ is defined to be the least cardinality of a $\sigma$-cover of $[\mu]^{<\theta}$ in $[\mu]^{<\kappa}$.
\end{definition}

We assume (\ref{ordering}) in order to avoid  uninteresting cases.  Since it is clear that
\begin{equation}
\cov(\mu,\kappa,\theta, 2)=\cov(\mu,\kappa,\theta,\aleph_0),
\end{equation}
we may as well assume that all four parameters are infinite cardinals. Note as well that $\cov(\mu,\kappa^+,\kappa^+,2)$ is just the cofinality of the structure $\left([\mu]^\kappa,\subseteq\right)$, so covering numbers refine this familiar notion.  Section 5 of Chapter II in~\cite{cardarith} contains a comprehensive list of basic properties satisfied by covering numbers.

Our work in this paper focuses on covering numbers in which the first two arguments are the same, so we take $\mu$ equal to $\kappa$. Such covering numbers describe the way in which $[\mu]^{<\theta}$ sits inside of $[\mu]^{<\mu}$.  This is interesting only in the case where $\mu$ is singular:  if $\mu$ is regular, then the initial segments of $\mu$ will automatically cover $[\mu]^{<\mu}$ so the covering number is just $\mu$. Furthermore, we assume $\sigma\leq\cf(\mu)<\theta$ to avoid similar trivialities.  Shelah has shown that general covering numbers can be computed from those where the first two components are the same, so our restriction is done without loss of generality.  It is the behavior at singular cardinals that is important.

\subsection{Pseudopowers}

One of Shelah's many surprising discoveries in cardinal arithmetic is that covering numbers can often be computed using pcf theory, and we exploit this link to obtain results like~(\ref{amusement}). The connection occurs through  {\em pseudopowers} of singular cardinals. Pseudopowers are best thought of as pcf-theoretic versions of cardinal exponentiation. Computing a pseudopower at a singular cardinal~$\mu$ involves examining the cardinals that can be represented in specific ways as the cofinality of reduced products of sets of regular cardinals cofinal in~$\mu$.  The following definition and subsequent discussion fix our vocabulary.

\begin{definition}
Suppose $\mu<\lambda$ are cardinals with $\mu$ singular and $\lambda$ regular. A {\em representation} of $\lambda$ at $\mu$ is a pair $(A, J)$ where
\begin{itemize}
\item $A$ is a progressive (that is, satisfying  $|A|<\min(A)$) cofinal subset of $\mu\cap\reg$
\sk
\item $J$ is an ideal on $A$ extending the bounded ideal $J^{\bd}[A]$
\sk
\item $\lambda$ is the {\em true cofinality} of the reduced product $\prod A/ J$, that is, there is a sequence $\langle f_\alpha:\alpha<\lambda\rangle$ in $\prod A$ such that
\begin{equation}
\alpha<\beta<\lambda\Longrightarrow f_\alpha<_J f_\beta,
\end{equation}
 and
 \begin{equation}
 (\forall g\in\prod A)(\exists\alpha<\lambda)\left[g<_J f_\alpha\right].
 \end{equation}
 We abbreviate this by writing $\lambda = \tcf(\prod A/ J)$.
\sk
\item $\lambda=\max\pcf(A)$.
\end{itemize}
A cardinal $\lambda$ is {\em representable at }$\mu$ if such a representation exists. More generally, given cardinals $\sigma<\theta\leq\mu$ with $\sigma$ regular, we say $\lambda$ is {\em $\Gamma(\theta,\sigma)$-representable at~$\mu$} if there is a representation $(A, J)$ of $\lambda$ at $\mu$ with $|A|<\theta$ and $\sigma$-complete ideal~$J$.\footnote{To highlight a point of potential confusion, note that $\sigma$ may be larger than $\aleph_1$, so we are not using ``$\sigma$-complete'' as a synonym for ``countably complete''.}
We say that $|A|$ is the {\em size} of the representation, and the {\em completeness} of the representation is the completeness of the ideal~$J$, that is, the least cardinal $\tau$ such that $\tau$ is not closed under unions of size $\tau$.
\end{definition}
When speaking about $\Gamma(\theta,\sigma)$-representability at a cardinal~$\mu$, we will always assume $\sigma\leq\cf(\mu)<\theta$ (as otherwise things degenerate), and that $\sigma$ and $\theta$ are regular.  This is reminiscent of the assumption~(\ref{ordering}), and it will guarantee that the associated pseudopowers obey some useful rules.

Moving on, we come  to the actual definition of the pseudopower operation:

\begin{definition}
Suppose $\mu$ is singular, and $\sigma\leq\cf(\mu)<\theta\leq\mu$ with $\sigma$ and $\theta$ regular.  We define
\begin{equation}
\PP_{\Gamma(\theta,\sigma)}(\mu):=\{\lambda:\text{ $\lambda$ is $\Gamma(\theta,\sigma)$-representable at $\mu$}\},
\end{equation}
and the $\Gamma(\theta,\sigma)$-pseudopower of~$\mu$, $\pp_{\Gamma(\theta,\sigma)}(\mu)$, is defined by
\begin{equation}
\pp_{\Gamma(\theta,\sigma)}(\mu):=\sup\PP_{\Gamma(\theta,\sigma)}(\mu).
\end{equation}
\end{definition}

There are a few notational variants used in the literature, all due to Shelah.  For example,
\begin{equation}
\pp_\theta(\mu):=\pp_{\Gamma(\theta^+,\aleph_0)}(\mu)
\end{equation}
(so in this case we are computing the supremum of all cardinals that can be represented at~$\mu$ using a set of size at most~$\theta$), and
\begin{equation}
\pp_{\Gamma(\sigma)}(\mu)=\pp_{\Gamma((\cf\mu)^+,\sigma)}(\mu),
\end{equation}
(which computes the supremum of the set of cardinals that can be represented at~$\mu$ using a $\sigma$-complete ideal on a set of size~$\cf(\mu)$).
Finally, the {\em pseudopower $\pp(\mu)$ of $\mu$} is defined by
\begin{equation}
 \pp(\mu):=\pp_{\cf(\mu)}(\mu).
 \end{equation}

A little discussion may help the reader digest the preceding definitions.  First, we note the obvious monotonicity property:  if $\sigma\leq\sigma'<\theta'\leq\theta$ are regular cardinals and $\mu$ is singular
with cofinality in the interval $[\sigma',\theta')$, then
\begin{equation}
\label{monotonicity}
\PP_{\Gamma(\theta',\sigma')}(\mu)\subseteq\PP_{\Gamma(\theta,\sigma)}(\mu).
\end{equation}
Given a singular cardinal~$\mu$, a well-known theorem of Shelah on the existence of scales at successors of singular cardinals provides us with a cofinal subset $A$ of $\mu\cap\reg$ such that
\begin{itemize}
\item $|A|=\cf(\mu)$, and
\sk
\item $\tcf(\prod A/ J^{\bd}[A])=\mu^+$, where $J^{\bd}[A]$ is the bounded ideal on $A$.
\sk
\end{itemize}
The bounded ideal is trivially $\cf(\mu)$-complete, so this means that $\mu^+$ is~$\Gamma(\cf(\mu))$-representable, and then an appeal to (\ref{monotonicity}) shows us that in fact $\mu^+$ is $\Gamma(\theta,\sigma)$-representable at~$\mu$ for all relevant $\sigma$ and $\theta$.  Thus,  $\pp_{\Gamma(\theta,\sigma)}(\mu)$ is always at least~$\mu^+$.

Lying a little deeper is a result of Shelah that $\PP_{\Gamma(\theta,\sigma)}(\mu)$ consists of an {\em interval} of regular cardinals, that is
\begin{equation}
\lambda\in\PP_{\Gamma(\theta,\sigma)}(\mu)\Longrightarrow [\mu^+,\lambda]\cap\reg\subseteq\PP_{\Gamma(\theta,\sigma)}(\mu).
\end{equation}
This is known as the {\sf No Holes Conclusion} (see 2.3 in Chapter~{II} of~\cite{cardarith}).

The interval of regular cardinals $\PP_{\Gamma(\theta,\sigma)}(\mu)$ enjoys some nice closure properties: if $A$ is a subset of $\PP_{\Gamma(\theta,\sigma)}(\mu)$ of cardinality less than~$\theta$, then
\begin{equation}
\label{pcfsigclose}
\pcfsig(A)\subseteq\PP_{\Gamma(\theta,\sigma)}(\mu),
\end{equation}
that is, the interval of regular cardinals $\PP_{\Gamma(\theta,\sigma)}(\mu)$ is closed under computing $\sigma$-complete pcf on sets of cardinality less than~$\theta$.\footnote{A cardinal $\lambda$ is in $\pcfsig(A)$ if there is a $\sigma$-complete ideal $J$ on $A$ with the true cofinality of $\prod A/ J$ equal to~$\lambda$. See, e.g., the chapter~\cite{AM}.}  This is a critical property for us, and it is usually expressed in terms of an {\em inverse monotonicity} property of $\Gamma(\theta,\sigma)$-pseudopowers:
\begin{proposition}[{\sf Inverse Monotonicity}]
\label{inversemonotonicity}
Suppose $\sigma\leq\cf(\mu)<\theta$ with $\sigma$ and $\theta$ regular.  If $\eta<\mu$ satisfies
\begin{itemize}
\item $\sigma\leq\cf(\eta)<\theta$, and
\sk
\item $\mu\leq\pp_{\Gamma(\theta,\sigma)}(\eta)$,
\end{itemize}
then
\begin{equation}
\PP_{\Gamma(\theta,\sigma)}(\mu)\subseteq\PP_{\Gamma(\theta,\sigma)}(\eta),
\end{equation}
and therefore
\begin{equation}
\pp_{\Gamma(\theta,\sigma)}(\mu)\leq\pp_{\Gamma(\theta,\sigma)}(\eta).
\end{equation}
\end{proposition}
The preceding result can be found as~$\otimes_1$ in Section~II.2 of~\cite{cardarith}. To see why this implies the closure property expressed in (\ref{pcfsigclose}), suppose  $A\subseteq\PP_{\Gamma(\theta,\sigma)}(\mu)$ satisfies $|A|<\theta$ and $J$ is $\sigma$-complete ideal on $A$ for which $\prod A/ J$ has true cofinality equal to some $\lambda$.  Let $\eta$ be the least cardinal such that $A\cap\eta\notin J$.  Then $\eta$ is singular with $\sigma\leq\cf(\eta)<\theta$, and $\lambda$ is $\Gamma(\theta,\sigma)$-representable at~$\eta$ by way of the pair $(A\cap\mu', J\restr A\cap\mu')$.   Inverse Monotonicity implies that $\lambda$ also $\Gamma(\theta,\sigma)$-representable at~$\mu$, and hence it is a member of~$\PP_{\Gamma(\theta,\sigma)}(\mu)$ by definition.

The final ingredient of the calculus of pseudopowers that we need is denoted {\em Continuity}, and can be found as $\otimes_2$ in Section~II.2 of~\cite{cardarith}:
\begin{proposition}[{\sf Continuity}]
\label{continuity}
Assume $\sigma\leq\cf(\mu)<\theta$ with $\sigma$ and $\theta$ regular, and let $\lambda$ be a regular cardinal greater than~$\mu$.  If $\lambda$ is $\Gamma(\theta,\sigma)$-representable at~$\eta$ for an unbounded set of singular cardinals $\eta<\mu$ (satisfying $\sigma\leq\cf(\eta)<\theta<\eta$), then $\lambda$ is also $\Gamma(\theta,\sigma)$-representable at~$\mu$.  In other words, if
\begin{equation}
\mu=\sup\{\eta<\mu:\sigma\leq\eta<\theta<\eta\text{ and }\lambda\in\PP_{\Gamma(\theta,\sigma)}(\eta)\},
\end{equation}
then
\begin{equation}
\lambda\in\PP_{\Gamma(\theta,\sigma)}(\mu).
\end{equation}
\end{proposition}
The proof is not difficult: one can paste together suitable representations of $\lambda$ at cardinals $\eta<\mu$ to obtain a representation at~$\mu$ itself.

\subsection{The $\cov$ vs. $\pp$ Theorem}

With the above material in hand, we can state the connection between covering numbers and pseudopowers. This theorem is due to Shelah (Theorem~5.4 of Chapter~{II} of~\cite{cardarith}); our paper~\cite{ei} gives another proof of the result, in addition to providing much more background about pseudopowers and their computation.

\begin{theorem}[{\sf The cov vs. pp Theorem}]
\label{covppthm}
 Suppose $\sigma<\theta$ are infinite regular cardinals, and $\mu$ is singular with $\sigma\leq\cf(\mu)<\theta$.
Then
\begin{equation}
\pp_{\Gamma(\theta,\sigma)}(\mu)\leq\cov(\mu,\mu,\theta,\sigma).
\end{equation}
If $\sigma>\aleph_0$ (so $\mu$ has uncountable cofinality) then
\begin{equation}
\label{1.20}
\ppgamma(\mu)=\cov(\mu,\mu,\theta,\sigma).
\end{equation}
\end{theorem}

The proof of (\ref{1.20}) is the hard part of the result, and the question of whether the uncountability of $\sigma$ is necessary to prove the equality is a major open question in pcf theory known as the ``cov vs. pp problem''.  Shelah has shown that in many cases this condition can be dropped, and that pcf theory in the neighborhood of a counterexample~$\mu$ must be very badly behaved.

\subsection{Overview revisited}

Given the above discussions, we hope the reader is now in a better position to understand the summary given in Subsection 1.1. For example, to obtain our sample equation~(\ref{amusement}), we prove that if a cardinal $\lambda$ is $\Gamma(\aleph_9,\aleph_2)$-representable at $\mu$ of cofinality $\aleph_6$, then either $\lambda$ is $\Gamma(\aleph_7,\aleph_2)$-representable at $\mu$, or it is $\Gamma(\aleph_9,\aleph_6)$-representable at $\mu$.   In other words, if $\lambda$ can be represented at a singular cardinal $\mu$ of cofinality $\aleph_6$ using an $\aleph_2$-complete ideal on a set of cardinality at most $\aleph_8$, then it can either be represented using an $\aleph_2$-complete ideal on a set of cardinality $\aleph_6$ (the minimum possible size of a representation at~$\mu$), or it can be represented on a set of cardinality at most $\aleph_8$ using an $\aleph_6$-complete ideal (the maximum possible completeness of a representation at~$\mu$). We then apply the cov vs. pp Theorem to draw conclusions about the corresponding covering numbers (see Corollary~\ref{cor7.1} and the subsequent discussion).\footnote{We also note here that the two options do not correspond cleanly to the possibilities of the Pseudopower Dichotomy: in the ``strong limit'' option of the dichotomy, we will actually get a representation of $\lambda$ using a $\cf(\mu)$-complete ideal on a set of size $\cf(\mu)$, and it is the other option of the dichotomy that is responsible for the complexity of~(\ref{amusement}).}

\subsection{Structure of the paper}
The remainder of the paper is structured as follows:
\begin{itemize}
\item In Section~2 we chain together some theorems of Shelah (scattered throughout several papers) concerning the existence of well-organized sets of generators in order to formulate a precise and user-friendly result (Corollary~\ref{nicecor}) that we then use to prove a generalization of the main result of~\cite{371}, tailored to $\sigma$-complete pcf.
\sk
\item Section~3 builds on this work and analyzes pseudopowers at singular cardinals that are {\em eventually $\Gamma(\theta,\sigma)$-closed}.  We are able to show that cardinals $\Gamma(\theta,\sigma)$-representable at such a~$\mu$ can be represented in a well-organized way.

\item Section~4 isolates the {\sf Pseudopower Dichotomy}, and then analyzes  the structure of $\PP_{\Gamma(\theta,\sigma)}(\mu)$ when the singular cardinal~$\mu$ is {\em NOT} eventually $\Gamma(\theta,\sigma)$-closed.
\sk
\item In Section~5, we use the previous sections to obtain results in {\sf ZFC} about equalities between various types of pseudopowers at a singular cardinal~$\mu$. In particular, we show
    \begin{equation}
    \pp_{\theta}(\mu) = \pp(\mu)+\pp_{\Gamma(\theta^+,\cf(\mu))}(\mu),
    \end{equation}
    and, for $\sigma<\cf(\mu)$,
    \begin{equation}
    \pp_{\Gamma(\theta,\sigma)}(\mu) = \pp_{\Gamma(\sigma)}(\mu) + \pp_{\Gamma(\theta,\sigma^+)}(\mu).
    \end{equation}
\sk
\item Section~6 uses recent work of Gitik~\cite{gitik} to provide complementary independence results related to the formulas derived in Section~5.
\sk
\item In Section~7 we map out consequences of these results for covering numbers, arriving at the formula (\ref{amusement}) and its relatives, and also discussing its consequences. We conclude with questions raised by this work.
\end{itemize}

\section{Reducing the size of a representations}

\subsection{On generators}

In this section, we prove a generalization of Theorem~1.1 of~\cite{371} that extends Shelah's result to $\sigma$-complete pcf. We prove the theorem by manipulating a suitably nice collection of pcf generators rather than by working directly with characteristic functions of models as in~\cite{371}.

The basic definitions follow, and we refer the reader to~\cite{AM} and \cite{burkemagidor} for more detailed discussion of these matters.
\begin{definition}
Let $A$ be a progressive set of regular cardinals.
\begin{enumerate}
\item If $\lambda\in\pcf(A)$, then a subset $B$ of $A$ is a {\em generator for $\lambda$ in $A$} if
\begin{equation}
J_{\leq\lambda}[A]= J_{<\lambda}[A]+ B.
\end{equation}
\item A {\em generating sequence} for $\pcf(A)$ is a sequence $\langle B_\lambda:\lambda\in\pcf(A)\rangle$ where $B_\lambda$ is a generator $\lambda$ in $A$.
\sk
\item More generally, if $C$ is a progressive subset of $\pcf(A)$ and $\Lambda$ is a subset of $\pcf(C)$,  we say that $\langle B_\lambda:\lambda\in\Lambda\rangle$ is a generating sequence for $\Lambda$ in $C$ if $B_\lambda$ is a generator for $\lambda$ in $C$ for all $\lambda\in \Lambda$, that is
    \begin{equation}
    J_{\leq\lambda}[C] = J_{<\lambda}[C] + B_\lambda
    \end{equation}
    for all $\lambda\in\Lambda$.
\end{enumerate}
\end{definition}

In (3), our assumptions imply $\pcf(C)\subseteq\pcf(A)$ so $\Lambda$ is a subset of $\pcf(A)$ as well. We will usually assume that $C$ includes $A$, and in this situation we know that if $B_\lambda$ is a generator for $\lambda$ in $C$, then $B_\lambda\cap A$ will be a generator for $\lambda$ in $A$.  Where important for clarity, we may write $B_\lambda[C]$ to emphasize that the corresponding set is a generator for $\lambda$ {\em in $C$}.

\begin{definition}
Let $A$ be a progressive set of regular cardinals, and suppose $\mathfrak{\bar{b}}$ is a generating sequence $\langle B_\lambda[C]:\lambda\in\Lambda\rangle$ for $\Lambda$ in $C$, where $\Lambda\subseteq\pcf(A)$ and $C$ is a progressive subset of $\pcf(A)$ containing $A$.
The sequence $\bar{\mathfrak{b}}$ is {\em transitive} if
\begin{equation}
\theta\in B_\lambda[C]\cap\Lambda\Longrightarrow B_\theta[C]\subseteq B_\lambda[C].
\end{equation}
\end{definition}

It is a fundamental result of pcf theory that for a progressive $A$, any $\lambda$ in $\pcf(A)$ has a corresponding generator in $A$, and thus we can always find a generating sequence $\langle B_\lambda[A]:\lambda\in\pcf(A)\rangle$ for $\pcf(A)$ in $A$.  Obtaining transitive generating sequences is a more complicated issue.  Shelah shows in Claim~6.7 of~\cite{430} that there is a transitive generating sequence for $A$ (not $\pcf(A)$!) in $A$ whenever $A$ is progressive. Thus, if $\pcf(A)$ happens to be progressive, then a transitive generating sequence for $\pcf(A)$ in $\pcf(A)$ will exist.  Abraham and Magidor prove something a little more general in Section~6 of~\cite{AM}.  A corollary of their presentation is that if $\kappa$ is a regular cardinal with $|A|<\kappa<\min(A)$ and $\Lambda$ is a subset of $\pcf(A)$ of cardinality at most $\kappa$, then we can find a transitive generating sequence for $\Lambda$ in $A$.

Our proof requires more than this, and we need some results of Shelah appearing in Claims~6.7A and B of~\cite{430}.  These are quite technical, so in the interest of readability we summarize our requirements in a ``black box'' result:

\begin{theorem}[{\sf Black Box}]
\label{blackbox}
Let $A$ be a progressive set of regular cardinals, and suppose $\kappa$ is a regular cardinal satisfying $|A|<\kappa<\min(A)$. Then for any sufficiently large regular $\chi$ and $x\in H(\chi)$, there is an elementary submodel $N$ of $H(\chi)$ of cardinality~$\kappa$ containing $A$ and $x$ and a sequence $\bar{\mathfrak{b}}=\langle B_\lambda:\lambda\in N\cap\pcf(A)\rangle$ such that
\begin{enumerate}
\item $\bar{\mathfrak{b}}$ is a transitive generating sequence for $N\cap\pcf(A)$ in $N\cap\pcf(A)$, and
\sk
\item given a sequence $\bar{A}=\langle A_\xi:\xi<\eta\rangle$ of sets and a regular cardinal $\sigma$ such that
\begin{enumerate}
\item $\bar{A}\in N$
\sk
\item $\eta$ and $\sigma$ are less than $\kappa$, and
\sk
\item $A_\xi$ is a subset of $N\cap\pcf(A)$ for each $\xi<\eta$,
\end{enumerate}
there is a family $\bar{C} = \langle C_\xi:\xi<\eta\rangle$ such that
\begin{itemize}
\item $\bar{C}\in N$,
\sk
\item $C_\xi$ is a subset of $\pcfsig(A_\xi)$ of cardinality less than $\sigma$, and
\sk
\item $A_\xi\subseteq\bigcup_{\lambda\in C_\xi}B_\lambda$
\end{itemize}
\end{enumerate}
\end{theorem}

Part (1) of the conclusion gives more than what Abraham and Magidor obtain in~\cite{AM}, as they produce a transitive generating sequence for $N\cap\pcf(A)$ in $A$, and not in the (possibly larger) set $N\cap \pcf(A)$. Said another way, the argument in~\cite{AM} provides sets $B_\lambda$ for $\lambda\in N\cap\pcf(A)$ that are subsets of $A$ and such that $B_\lambda$ generates the ideal $J_{\leq\lambda}[A]$ over $J_{<\lambda}[A]$.  The stronger version we need provides generators that function in the progressive set $N\cap\pcf(A)$, rather than just in $A$, with the corresponding $B_\lambda$ generating $J_{\leq\lambda}[N\cap\pcf(A)]$ over $J_{<\lambda}[N\cap\pcf(A)]$.  This extra power will be important for us, and it follows from parts (1) and (2) of Claim 6.7B of~\cite{430}.

Part (2) of the conclusion says that even though the generators we produce may not be in $N$, they do interact well with $N$ in the sense that pcf compactness arguments work as long as we are trying to cover sets in $N$.   This follows from part (3) of Claim 6.7B of~\cite{430}, and it is arranged by showing that the generators produced have suitable ``internal reflections'' inside of $N$.

\subsection{Reducing the size of a representation}
We turn now to the main result of this section: a relative of Theorem~1.1 from Chapter~VIII of~\cite{cardarith}.

\begin{theorem}
\label{practice}
Suppose $A=\bigcup_{\xi<\eta}A_\xi$ is a progressive set of regular cardinals, and suppose $\sigma$ is a regular cardinal. Suppose
\begin{equation}
\lambda=\max\pcf(A),
\end{equation}
and
\begin{equation}
\label{eqn4}
\lambda\in\pcfsig(A)\setminus\bigcup_{\xi<\eta}\pcfsig(A_\xi).
\end{equation}
Then can find a subset $C$ of $\bigcup_{\xi<\eta}\pcfsig(A_\xi)$ such that
 \begin{equation}
 |C|\leq\eta,
 \end{equation}
 and
 \begin{equation}
 \lambda\in\pcfsig(C).
 \end{equation}
\end{theorem}

This is a template for what we might call a reduction-in-size theorem.  The point is that we are replacing the set $A$ by set $C$ whose size is under our control, and doing it in such a way that $\lambda$ is still captured by $\pcfsig(C)$.

\begin{proof}
Looking at our assumptions, it is clear that we may assume that $\eta$ is at most the cardinality of $A$.  The assumptions also imply that $\eta$ is at least $\sigma$, as otherwise  $\pcfsig(A)$ would simply be the union of the various $\pcfsig(A_\xi)$ for $\xi<\eta$, and this would violate~(\ref{eqn4}).  Thus, we can assume $\sigma\leq\eta\leq |A|$.

We may also assume that $|A|^+$ is strictly less than $\min(A)$, and so (by setting $\kappa=|A|^+$) we can find a model $N$ as in Theorem~\ref{blackbox} containing $A$ and the sequence $\langle A_\xi:\xi<\eta\rangle$.   Note that since $\kappa+1\subseteq N$, we know that $A$ is a subset of $N$ and each individual element $A_\xi$ of our sequence will be in $N$ as well.  Since $|A_\xi|<\kappa$ for each $\xi$, it follows that $A_\xi$ will be a subset of $N\cap\pcf(A)$ too.

Let $\langle B_\lambda:\lambda\in N\cap\pcf(A)\rangle$ be the transitive sequence from Theorem~\ref{blackbox}. Conclusion (2) of this theorem tells us there is a sequence $\langle C_\xi:\xi<\eta\rangle$ in $N$ such that each $C_\xi$ is a subset of $\pcfsig(A_\xi)$ of cardinality less than~$\sigma$ and
\begin{equation}
A_\xi\subseteq\bigcup_{\lambda\in C_\xi}B_\lambda.
\end{equation}
We set
\begin{equation}
C=\bigcup_{\xi<\eta}C_\xi.
\end{equation}
Then $|C|\leq\eta\cdot\sigma=\eta$, and $C$ is definable from the sequence $\langle C_\xi:\xi<\eta\rangle$ in~$N$.

We now show that $\lambda$ is in $\pcfsig(C)$, which will finish the proof.  To do this, we apply conclusion (2) of Theorem~\ref{blackbox} again, this time to the single set $C\in N$. We obtain a set $D\in N$ of cardinality less than~$\sigma$ such that
\begin{equation}
\label{tallis}
D\subseteq\pcfsig(C)\subseteq \pcfsig(A)
\end{equation}
and
\begin{equation}
\label{eqn3}
C\subseteq\bigcup_{\lambda\in D} B_\lambda.
\end{equation}
This is the crucial point where we need for our generators to work in $N\cap\pcf(A)$ rather than just in the potentially smaller set~$A$, as we are applying a compactness argument to cover a subset of $N\cap\pcf(A)$.

We know that $\max\pcf(C)$ is at most~$\lambda$ because
\begin{equation}
\max\pcf(\pcf(A))=\max\pcf(A).
\end{equation}
Thus, if $\lambda$ fails to be in~$\pcfsig(C)$, it must be the case that
\begin{equation}
D\subseteq\pcfsig(A)\cap\lambda,
\end{equation}
that is, all members of $D$ must be less than $\lambda$.

We now have an untenable situation.  Since $\lambda$ is in $\pcfsig(A)$, the $\sigma$-complete ideal on $A$ generated by $J_{<\lambda}[A]$ must be a proper ideal.  But for $\epsilon\in D$, we know $B_\epsilon\cap A$ is in the ideal $J_{<\lambda}[A]$, and since $|D|<\sigma$ it follows that $A$ cannot be contained in the union of the sets $B_\epsilon$ for $\epsilon$ in $D$, that is,
\begin{equation}
\label{fork}
A\nsubseteq \bigcup_{\epsilon\in D}B_\epsilon.
\end{equation}
On the other hand, though, for each $\zeta\in A$ there is a $\tau\in C$ with $\zeta\in B_\tau$, and there is a corresponding $\epsilon\in D$ with $\tau\in B_\epsilon$.  By transitivity, we know $B_\tau\subseteq B_\epsilon$ and therefore $\zeta\in B_\epsilon$. Thus
\begin{equation}
A\subseteq\bigcup_{\epsilon \in D} B_\epsilon,
\end{equation}
which contradicts (\ref{fork}). Thus,  $\lambda$ must be in $\pcfsig(C)$.
\end{proof}

\section{Eventually $\Gamma(\theta,\sigma)$-closed cardinals}

\subsection{Motivation} In this section we work with cardinals that are strong limits in a sense measured by pseudopowers.  Our aim is to generalize one of the main results of Chapter~VIII of {\em Cardinal Arithmetic}, where Shelah addresses basic questions about improving representations of cardinals. Simply recalling a small part of what Shelah establishes provides us with a good starting point.

\begin{theorem}[Shelah~\cite{371}]
\label{371}
Suppose $\aleph_0<\cf(\mu)\leq\theta<\mu$, and for every sufficiently large $\eta<\mu$,
\begin{equation}
\label{highlight}
\cf(\eta)\leq\theta\Longrightarrow \pp_\theta(\eta)<\mu.
\end{equation}
Then $\PP_\theta(\mu)=\PP_{\Gamma(\cf(\mu))}(\mu)$.  In fact, any $\lambda\in\PP_\theta(\mu)$ can be represented as the true cofinality of $\prod A/J^{\bd}[A]$ where $A$ is a cofinal subset of $\mu\cap\reg$ of cardinality~$\cf(\mu)$.
\end{theorem}

The above is part of Corollary~1.6 on page 321 of~\cite{cardarith}, and we will shortly derive it  from our own work in this section.  For now, we wish to highlight the assumption~(\ref{highlight}): it expresses that~$\mu$ is (in a weak sense) a type of strong limit cardinal, and we will be working with such assumptions a lot in this section.  The conclusion of the above theorem can be described informally in terms of upgrading the  representation of~$\lambda$:  we are able to move from  an arbitrary representation of~$\lambda$ based on a set of cardinality at most $\theta$ to one based on a set of cardinality $\cf(\mu)$, with the added bonus that the ideal used in the representation is as simple as possible.  The results we prove in this section will have a similar flavor.

We start with our main hypothesis, a definition that is natural given the preceding discussion.

\begin{definition}
Let $\sigma$ and $\theta$ be regular cardinals, and suppose $\mu$ is singular with $\sigma\leq\cf(\mu)<\theta<\mu$.
\begin{enumerate}
\item  We say that $\mu$ is {\em eventually $\Gamma(\theta,\sigma)$-closed} if
for all sufficiently large $\nu<\mu$, if $\nu$ is singular with $\sigma\leq\cf(\nu)<\theta$ then $\pp_{\Gamma(\theta,\sigma)}(\nu)<\mu$.

\sk

\item We say that $\mu$ is  $\Gamma(\theta,\sigma)$-closed beyond~$\eta$ if $\eta<\mu$ and the above holds for all $\nu$ between $\eta$ and $\mu$.
\end{enumerate}
\end{definition}

Shelah has looked at such concepts in a more general setting. In particular, the third section of~\cite{410} briefly examines the idea of ``pcf inaccessibility''. The following lemma shows that our definition and his approach are essentially the same.

\begin{lemma}
\label{closed}
Suppose $\sigma<\theta$ are regular cardinals, and $\mu$ is a singular cardinal satisfying $\sigma\leq\cf(\mu)<\theta<\mu$. Then the following conditions are equivalent:
\begin{enumerate}
\item $\mu$ is eventually $\Gamma(\theta,\sigma)$-closed.
\item There is an $\eta<\mu$ such that if $A$ is a set of regular cardinals from the interval $(\eta,\mu)$ bounded below $\mu$, then
    \begin{equation}
    \pcfsig(A)\subseteq\mu.
    \end{equation}
\end{enumerate}
\end{lemma}
\begin{proof}
Assume (1), and choose $\eta<\mu$ such that $\mu$ is $\Gamma(\theta,\sigma)$-closed beyond~$\eta$.  Given a set $A$ satisfying the assumptions of~(2), suppose $\lambda$ is in $\pcfsig(A)\setminus\mu$.   By passing to the generator $B_\lambda[A]$, we may assume $\lambda$ is $\max\pcf(A)$, and so if $J$ is the $\sigma$-complete ideal on $A$ generated by $J_{<\lambda}[A]$ then
\begin{equation}
\lambda = \tcf\left(\prod A/ J\right).
\end{equation}
Now let $\xi$ be least ordinal with $A\cap\xi\notin J$. Our assumptions imply that $\xi$ is a singular cardinal of cofinality at most $|A|<\theta$, and $\lambda$ is $\Gamma(\theta,\sigma)$-representable at~$\xi$.  Since $\eta<\xi<\mu$, we have a contradiction.  The proof that (2) implies (1) is even easier: if $\mu$ is not eventually $\Gamma(\theta,\sigma)$-closed, then the associated representations show that (2) must fail.
\end{proof}

\subsection{Basic representation theorems} Turning to the main topic of this section, we begin with an application of Theorem~\ref{practice} to the question of representation.

\begin{theorem}
\label{theoremclosed}
Suppose $\sigma\leq\cf(\mu)<\theta<\mu$ with $\sigma$ and $\theta$ regular.  If $\mu$ is eventually $\Gamma(\theta,\sigma)$-closed, then
\begin{equation}
\PP_{\Gamma(\theta,\sigma)}(\mu) = \PP_{\Gamma(\sigma)}(\mu).
\end{equation}
\end{theorem}

Before we give the proof, note that this result is a natural counterpart to Theorem~\ref{371}, with the only difference being the inclusion of the parameter~$\sigma$. This theorem asks for a weaker closure condition than~(\ref{highlight}) in the situation where $\sigma$ is uncountable, but it also has a weaker conclusion: it says only that if $\mu$ is eventually $\Gamma(\theta,\sigma)$-closed, then any cardinal representable at~$\mu$ via a $\sigma$-complete ideal on a set of cardinality less than $\theta$ is in fact representable at~$\mu$ via a $\sigma$-complete ideal on a set of cardinality~$\cf(\mu)$, the minimum possible, but the proof does not let us obtain a representation using the bounded ideal.

\begin{proof}
Suppose $A$ and $J$ witness that the cardinal $\lambda$ is $\Gamma(\theta,\sigma)$-representable at~$\mu$.
Since we may remove an initial segment of~$A$ if necessary, we may assume that $\mu$
 is $\Gamma(\theta,\sigma)$-closed beyond $\min(A)$, and by restricting to a suitable generator if necessary, we may assume
 that $\lambda= \max\pcf(A)$ as well.  Given $\langle \mu_\alpha:\alpha<\kappa\rangle$  increasing and cofinal in $\mu$, if we define
\begin{equation}
A_\alpha:= A\cap\mu_\alpha,
\end{equation}
then the hypotheses of Theorem~\ref{practice} are satisfied by $\langle A_\alpha:\alpha<\cf(\mu)\rangle$ and $\lambda$, and we obtain a set
\begin{equation}
\label{small}
C\subseteq\bigcup_{\alpha<\cf(\mu)}\pcfsig(A_\alpha)
\end{equation}
of cardinality less than $\theta$ such that
\begin{equation}
\lambda\in\pcfsig(C).
\end{equation}
Now let $I$ be the $\sigma$-complete ideal generated by $J_{<\lambda}[C]$.  We claim that $C$ and $I$ witness that $\lambda$ is $\Gamma(\theta)$-representable at $\mu$.

Certainly $C$ is a subset of $\mu\cap\reg$ because of (\ref{small}) and Lemma~\ref{closed}. We know $I$ is a proper ideal on $C$ because $\lambda$ is in $\pcfsig(C)$, and
\begin{equation}
\lambda = \tcf\left(\prod C/ I\right)
\end{equation}
because $\lambda=\max\pcf(C)$ and $I$ extends $J_{<\lambda}[C]$.
To finish, we need to show that $C$ is unbounded in $\mu$ and $I$ includes all initial segments of $C$, but this follows from Lemma~\ref{closed} as well: any subset of $C$ bounded below $\mu$ must be in $I$ because $\lambda$ is greater than $\mu$.
\end{proof}

We can do better than this, though, and the next lemma lies is the heart of our results. It shows that with assumptions similar to those used in Theorem~\ref{theoremclosed}, we are able to obtain representations of cardinals that are ``well-organized''. This will then allow us to show that the corresponding ideals in the representation satisfy stronger completeness conditions.

\begin{lemma}
\label{structurelemma}
Assume $\mu$ is eventually $\Gamma(\theta,\sigma)$-closed, where $\sigma$ and $\theta$ are regular, and $\sigma\leq\cf(\mu)<\theta$.  Suppose  $\lambda$ is $\Gamma(\theta,\tau)$-representable at~$\mu$ for some regular cardinal $\tau$ in the interval $[\sigma,\cf(\mu)]$.  Then we can find a cardinal $\sigma^*<\sigma$ and a set
\begin{equation}
C=\{\lambda_\varsigma^\alpha:\alpha<\cf(\mu)\text{ and }\varsigma<\sigma^*\}
\end{equation}
of regular cardinals less than $\mu$ such that
\begin{itemize}
\item $\pcfsig\{\lambda_\varsigma^\beta:\beta<\alpha\text{ and }\varsigma<\sigma^*\}\subseteq\mu$ for each $\alpha<\cf(\mu)$, and
\sk
\item if $X$ is an unbounded subset of $\cf(\mu)$ then
\begin{equation}
\lambda=\max\pcf(\{\lambda^\alpha_\varsigma:\alpha\in X\text{ and }\varsigma<\sigma^*\})
\end{equation}
and
\begin{equation}
\label{eqn11}
\lambda\in\pcftau(\{\lambda^\alpha_\varsigma:\alpha\in X\text{ and }\varsigma<\sigma^*\}).
\end{equation}
\end{itemize}
\end{lemma}

The above is essentially a generalization of Shelah's result Theorem~\ref{371}. We will discuss this after the proof, and even show how his result follows easily from the lemma.

\begin{proof}
Suppose $A$ and $J$ are a $\Gamma(\theta,\sigma)$-representation of~$\lambda$ at~$\mu$.  Just as in the proof of Theorem~\ref{theoremclosed}, we may assume $\mu$ is $\Gamma(\theta,\sigma)$-closed beyond $\min(A)$ and that $\lambda$ is $\max\pcf(A)$.  Let $\langle \mu_\alpha:\alpha<\cf(\mu)\rangle$
be an increasing sequence cofinal in~$\mu$, and let $A_\alpha$ be $A\cap\mu_\alpha$.

We implement the argument of Theorem~\ref{practice} and make use of transitive generators.  To do this, we assume (without loss of generality) that $|A|<\kappa=\cf(\kappa)<\min(A)$, and let $N$ and $\bar{\mathfrak{b}}$ be as inTheorem~\ref{blackbox} with $N$ containing all objects under discussion in the preceding paragraph.  By properties of $\bar{\mathfrak{b}}$, in the model $N$ there is a sequence of sets $C_\alpha$ for $\alpha<\cf(\mu)$ such that
\begin{itemize}
\item $C_\alpha$ is a subset of $\pcfsig(A_\alpha)$ of cardinality less than $\sigma$, and
\sk
\item $A_\alpha\subseteq\bigcup_{\lambda\in C_\alpha} B_\lambda$.
\end{itemize}
Since the sequence $\langle A_\alpha:\alpha<\cf(\mu)\rangle$ is increasing, it follows that whenever $X$ is an unbounded subset of~$\cf(\mu)$
we have
\begin{equation}
A\subseteq\bigcup_{\alpha\in X}\bigcup_{\lambda\in C_\alpha}B_\lambda,
\end{equation}
and so the transitivity arguments from the proof of Theorem~\ref{practice} tells us
\begin{equation}
\label{eqn9}
\lambda=\max\pcf\left(\bigcup_{\alpha\in X}C_\alpha\right)
\end{equation}
and
\begin{equation}
\label{eqn8}
\lambda\in\pcftau\left(\bigcup_{\alpha\in X}C_\alpha\right).
\end{equation}

Since each $C_\alpha$ has cardinality less than $\sigma<\cf(\mu)$, we can (by passing to an unbounded subset of $\cf(\mu)$) assume each $C_\alpha$ is of some fixed cardinality $\sigma^*<\sigma$, say
\begin{equation}
C_\alpha = \{\lambda^\alpha_\varsigma:\varsigma<\sigma^*\},
\end{equation}
and we are done by letting $C$ be the union of the sets $C_\alpha$.
\end{proof}

Note that we do not claim that the sets $C_\alpha$ are disjoint, and they may very well overlap.  It is helpful to visualize $C$ as an array of regular cardinals with $\cf(\mu)$ rows and $\sigma^*$ columns.   In this interpretation, row $\alpha$ corresponds to $C_\alpha$, and we get a corresponding column for each fixed $\varsigma<\sigma^*$.

The pcf structure of $C$ transfers to the index set $\Lambda = \cf(\mu)\times\sigma^*$ in the natural way, and we may define an ideal $J$ on
$\Lambda$ by
\begin{equation}
\mathcal{X}\in J\Longleftrightarrow \max\pcf\left(\{\lambda^\alpha_\varsigma:(\alpha,\varsigma)\in \mathcal{X}\}\right)<\lambda^*.
\end{equation}
With this point of view, we see that for $X\subseteq\cf(\mu)$,
\begin{equation}
\{(\alpha,\varsigma):\alpha\in X\text{ and }\varsigma<\sigma^*\}\in J\Longleftrightarrow \text{$X$ is bounded in $\cf(\mu)$}.
\end{equation}
Note as well that the $\tau$-complete ideal generated by~$J$ is a proper ideal because of~(\ref{eqn8}).

It is helpful to look back and compare our situation with the conclusion of~Theorem~\ref{371}.  We have not managed to represent $\lambda$ as the true cofinality of a product of cardinals modulo the bounded ideal, but we have come close! What goes wrong is that
the rows $C_\alpha$ in our array are not singletons, and instead all we know is that they all have a fixed cardinality~$\sigma^*$ less than~$\sigma$. What Shelah does in~\cite{371} is note that if this cardinality happens to be finite, then we can improve the situation and get the ideal to consist of just the bounded sets:

\begin{corollary}[Theorem~\ref{371}, due to Shelah~\cite{371}]
\label{3.6}
Suppose $\aleph_0<\cf(\mu)\leq\theta<\mu$ and $\mu$ is eventually $\Gamma(\theta^+,\aleph_0)$-closed, that is, for all sufficiently large $\nu<\mu$,
\begin{equation}
\cf(\nu)\leq\theta\Longrightarrow \pp_\theta(\nu)<\mu,
\end{equation}
Then any member of $\PP_{\theta}(\mu)$ has a representation of the form $(C, J^{\bd}[C])$ where $C$ is unbounded in $\mu\cap\reg$
 of order-type~$\cf(\mu)$ and $J^{\bd}[C]$ is the ideal of bounded subsets of~$C$.
\end{corollary}
 \begin{proof}
Suppose $\lambda=\tcf\prod A/ J$ where $A$ is cofinal in~$\mu$ of cardinality at most~$\theta$ and $J$ is an ideal on $A$ extending the bounded ideal.  We apply Lemma~\ref{structurelemma} to $\Gamma(\theta^+,\aleph_0)$ and obtain $n<\omega$ and $C=\langle \lambda^\alpha_i:i\leq n\rangle$
 such that
 \begin{itemize}
 \item $\sup\pcf\{\lambda^\beta_i:\beta<\alpha\text{ and }i\leq n\}$ is less than $\mu$ for each $\alpha<\cf(\mu)$, and
 \sk
 \item $\lambda=\max\pcf(\{\lambda^\alpha_i:\alpha\in X\text{ and }i\leq n\})$ for any unbounded $X\subseteq\cf(\mu)$.
 \end{itemize}
We now claim there is an $i\leq n$ and an unbounded $X\subseteq\cf(\mu)$ such that
\begin{equation}
\lambda=\max\pcf\{\lambda^\alpha_i:\alpha\in Y\}
\end{equation}
for every unbounded $Y\subseteq X$.  This is enough, as letting $D=\{\lambda^\alpha_i:i\in Y\}$, it follows easily that the ideal $J_{<\lambda}[D]$ consists of the bounded ideal $J^{\bd}[D]$.

We establish this by contradiction:  if there are no such $i$ and $X$, then for every $i\leq n$ and every unbounded $X\subseteq\cf(\mu)$
there is an unbounded $Y\subseteq X$ such that
\begin{equation}
\max\pcf\{\lambda^\alpha_i:\alpha\in Y\}<\lambda.
\end{equation}
Working by induction, we find a single unbounded $X\subseteq \cf(\mu)$ such that
\begin{equation}
(\forall i\leq n)\left(\max\pcf\{\lambda^\alpha_i:\alpha\in X\}\right)<\lambda.
\end{equation}
But then
\begin{equation}
\max\pcf\{\lambda^\alpha_i:\alpha\in X\text{ and }i\leq n\}<\lambda,
\end{equation}
as any ultrafilter on this set must contain one of the columns, and we have a contradiction.
\end{proof}

We note that Shelah obtains even nicer representations in Chapter~VIII of {\em Cardinal Arithmetic}, but the above observation is at the heart of his argument.

\subsection{Better representations} Moving on to the more general situation, we can apply Lemma~\ref{structurelemma} to obtain the following improvement of Theorem~\ref{theoremclosed}.

\begin{theorem}
\label{theoremgeneral}
Suppose $\mu$ is eventually $\Gamma(\theta,\sigma)$-closed for some regular $\sigma$ and $\theta$ with
\begin{equation}
\sigma<\cf(\mu)<\theta<\mu,
\end{equation}
and $\tau<\upsilon$ are regular cardinals in the interval $[\sigma,\cf(\mu)]$.  If
\begin{equation}
\label{eqn10}
\tau\leq\rho<\upsilon\Longrightarrow \cov(\rho,\rho,\sigma,2)<\cf(\mu)
\end{equation}
then
\begin{equation}
\PP_{\Gamma(\theta,\tau)}(\mu)=\PP_{\Gamma(\upsilon)}(\mu).
\end{equation}
\end{theorem}

Before presenting a proof of this theorem, we note that condition~(\ref{eqn10}) holds if $\upsilon$ is less than $\tau^{+\omega}$, so at the very least, the theorem allows us to find representations that are ``more complete'', while simultaneously ensuring that the size of the set involved is at most~$\cf(\mu)$. In order to do this, we need to make the stronger assumption that $\sigma$ is strictly less than $\cf(\mu)$, and this is what allows us to achieve greater completeness in our representation.

\begin{proof}
Suppose $\lambda$ is $\Gamma(\theta,\tau)$-representable at~$\mu$, and apply Lemma~\ref{structurelemma} to obtain a cardinal $\sigma^*<\sigma$ and a set $C$ of regular cardinals $\lambda^\alpha_\varsigma$ (for $\alpha<\cf(\mu)$ and $\varsigma<\sigma^*$) as there. In particular, if $X$ is any unbounded subset of~$\cf(\mu)$, then
\begin{equation}
\lambda\in\pcftau\{\lambda^\alpha_\varsigma:\alpha\in X\text{ and }\varsigma<\sigma^*\}.
\end{equation}
It suffices to prove that the $\upsilon$-complete ideal on $C$ generated by $J_{<\lambda^*}[C]$ is proper.

Suppose by way of contradiction this is not the case.  Then there is a least cardinal $\rho<\upsilon$ such that for some sequence of sets $\langle D_i:i<\rho\rangle$ and unbounded $X\subseteq\cf(\mu)$, we have
\begin{itemize}
\item $D_i\in J_{<\lambda^*}[C]$ for each $i<\rho$, and
\sk
\item $\{\lambda^\alpha_\varsigma:\alpha\in X\text{ and }\varsigma<\sigma^*\}\subseteq \bigcup_{i<\rho}D_i$.
\end{itemize}
Clearly $\tau$ must be less than or equal to $\rho$, and by shrinking $C$ if necessary we may assume there are such sets $D_i$ for $i<\rho$ with
\begin{equation}
C\subseteq\bigcup_{i<\rho}D_i.
\end{equation}
Our assumptions tell us that $\cov(\rho,\rho,\sigma, 2)$ is less than~$\cf(\mu)$, so there is a family ~$\mathcal{P}$ such that
\begin{itemize}
\item $\mathcal{P}\subseteq [\rho]^{<\rho}$,
\sk
\item $|\mathcal{P}|<\cf(\mu)$, and
\sk
\item $(\forall U\in [\rho]^{<\sigma})(\exists V\in \mathcal{P})[U\subseteq V]$.
\end{itemize}
For each $\alpha<\cf(\mu)$, there is a set $Y_\alpha\in [\rho]^{<\sigma}$  such that
\begin{equation}
C_\alpha=\{\lambda^\alpha_\varsigma:\varsigma<\sigma^*\}\subseteq\bigcup_{i\in Y_\alpha} D_i,
\end{equation}
Since $|\mathcal{P}|<\cf(\mu)$, there is a single set $Y\in\mathcal{P}$ and an unbounded subset $X$ of~$\cf(\mu)$ such that
\begin{equation}
\label{eqn12}
\bigcup_{\alpha\in X}C_\alpha\subseteq\bigcup_{i\in Y} D_i.
\end{equation}
But now we have a contradiction, as we have shown that $\bigcup_{\alpha\in X}C_\alpha$ can be covered by $|Y|<\rho$ sets  from $J_{<\lambda^*}[C]$.
\end{proof}

Again, note that Theorem~\ref{theoremclosed} can only tell us that $\PP_{\Gamma(\theta,\tau)}(\mu)=\PP_{\Gamma(\tau)}(\mu)$ under the same assumptions as Theorem~\ref{theoremgeneral}. The improvement we obtain here is that we find a representation using an  $\upsilon$-complete ideal on a set of size~$\cf(\mu)$ rather than simply a $\tau$-complete one.

We close this section with the following corollary, which gives us a little information about the assumptions we make relating $\sigma$, $\tau$, and $\upsilon$.

\begin{corollary}
Suppose $\mu$ is eventually $\Gamma(\sigma)$-closed, where $\sigma<\cf(\mu)$.  If $\tau<\upsilon$ are regular cardinals in the interval $[\sigma,\cf(\mu)]$ and
\begin{equation}
\PP_{\Gamma(\upsilon)}(\mu)\subsetneqq \PP_{\Gamma(\tau)}(\mu),
\end{equation}
then there is a singular cardinal $\rho$ of cofinality less than $\sigma$ such that
\begin{equation}
\tau<\rho<\upsilon\leq\cf(\mu)<\cov(\rho,\rho,\sigma, 2),
\end{equation}
hence
\begin{equation}
\tau<\rho<\upsilon\leq\cf(\mu)<\cf([\rho]^{<\sigma},\subseteq)\leq\rho^{<\sigma}.
\end{equation}
\end{corollary}

If we work with cardinals whose cofinality is $\aleph_n$ for some $n$, then the above situation cannot happen because there is no place for such a cardinal $\rho$. Thus, the following is immediate.

\begin{corollary}
Suppose $\mu$ is singular with $\aleph_0<\cf(\mu)<\aleph_\omega$.  If $\mu$ is eventually $\Gamma(\theta,\sigma)$-closed for some $\sigma<\cf(\mu)$, then
\begin{equation}
\PP_{\Gamma(\theta,\sigma)}(\mu)=\PP_{\Gamma(\cf(\mu))}(\mu).
\end{equation}
\end{corollary}

\section{The Pseudopower Dichotomy}

\subsection{Introducing the dichotomy} In the preceding section we analyzed the behavior of $\pp_{\Gamma(\theta,\tau)}(\mu)$ where $\mu$ is a singular cardinal that is eventually $\Gamma(\theta,\sigma)$-closed and $\sigma\leq\tau\leq\cf(\mu)$.  The analysis showed that if a cardinal is $\Gamma(\theta,\tau)$-representable at $\mu$, then it is has a $\sigma$-complete representation using the minimum possible size, $\cf(\mu)$, and the completeness of the representation can be increased if $\sigma$ is strictly less than $\tau$. In this section, we will analyze what happens when~$\mu$ is {\em not} eventually $\Gamma(\theta,\sigma)$-closed.  We start with the following result which formulates a fundamental dichotomy about singular cardinals.  This is not a new idea (in fact, the result we present is a relative of Fact~1.9 in~\cite{371}), and Shelah has made and used similar observations in other places. However, such dichotomies are extraordinarily useful for proving theorems in {\sf ZFC}, as they allow one to analyze situations by breaking into cases where each option carries non-trivial information.  We will refer to this result as {\em the Pseudopower Dichotomy}.

\begin{lemma}[{\sf Pseudopower Dichotomy}]
Suppose $\mu$ is a singular cardinal, and let $\sigma<\theta$ be regular cardinals with $\sigma\leq\cf(\mu)<\theta$.  Then exactly one of the following statements hold:
\begin{description}
\item[Option 1] $\mu$ is eventually $\Gamma(\theta,\sigma)$-closed.
\sk
\item[Option 2] $\mu$ is a limit of eventually $\Gamma(\theta,\sigma)$-closed cardinals $\eta$ for which
\begin{equation}
\sigma\leq\cf(\eta)<\theta,
\end{equation}
and
 \begin{equation}
 \label{option2}
 \PP_{\Gamma(\theta,\sigma)}(\mu)=\PP_{\Gamma(\theta,\sigma)}(\eta)\setminus\mu^+.
 \end{equation}
\end{description}
\end{lemma}

\begin{proof}
It is clear that the two options are mutually exclusive as $\mu^+$ is always $\Gamma(\theta,\sigma)$-representable at $\mu$.   Suppose that Option~1 fails for a cardinal~$\mu$. By definition, this means that for each $\xi<\mu$ there is a singular cardinal $\eta$ greater than $\xi$ such that
\begin{equation}
\label{eqn4.1}
\sigma\leq\cf(\eta)<\theta<\eta<\mu,
\end{equation}
and
\begin{equation}
\mu\leq \pp_{\Gamma(\theta,\sigma)}(\eta).
\end{equation}
By Inverse Monotonicity (Proposition~\ref{inversemonotonicity}), this implies
\begin{equation}
\PP_{\Gamma(\theta,\sigma)}(\mu)\subseteq\PP_{\Gamma(\theta,\sigma)}(\eta).
\end{equation}
Note as well that if $\eta$ is the LEAST such cardinal above $\xi$, then $\eta$ is $\Gamma(\theta,\sigma)$-closed beyond~$\xi$.  Thus, $\mu$ is a limit of eventually $\Gamma(\theta,\sigma)$-closed cardinals $\eta$ for which~(\ref{eqn4.1}) is true, and for which
\begin{equation}
\label{even}
\PP_{\Gamma(\theta,\sigma)}(\mu)\subseteq\PP_{\Gamma(\theta,\sigma)}(\eta).
\end{equation}

Let $\mathcal{X}$ be the collection of such cardinals $\eta<\mu$, and for each
 $\eta$ in $\mathcal{X}$, we define
\begin{equation}
\mathcal{Y}_\eta = \PP_{\Gamma(\theta,\sigma)}(\eta)\setminus\mu^+.
\end{equation}
If $\eta$ is in $\mathcal{X}$, then  $\mathcal{Y}_\eta$ is an non-empty interval of regular cardinals with minimum~$\mu^+$. Furthermore,  if $\eta<\nu$ in $\mathcal{X}$ then by Inverse Monotonicity we have
\begin{equation}
\mu^+\leq\sup(\mathcal{Y}_\nu)=\pp_{\Gamma(\theta,\sigma)}(\nu)\leq \pp_{\Gamma(\theta,\sigma)}(\eta)=\sup(\mathcal{Y}_\eta).
\end{equation}
Thus, the sequence $\langle\mathcal{Y}_\eta:\eta\in\mathcal{X}\rangle$ must eventually stabilize, say with value $\mathcal{Y}$.

If $\lambda\in\mathcal{Y}$, then by an application of the continuity property of Proposition~\ref{continuity} we know $\lambda\in\PP_{\Gamma(\theta,\sigma)}(\mu)$, and so for all sufficiently large $\eta\in\mathcal{X}$,
we have
\begin{equation}
\PP_{\Gamma(\theta,\sigma)}(\eta)\setminus\mu^+ = \mathcal{Y}\subseteq\PP_{\Gamma(\theta,\sigma)}(\mu).
\end{equation}
Combining this with (\ref{even}) establishes that Option 2 of the Pseudopower Dichotomy holds.
\end{proof}

As a corollary,  we have the following result  which we will use later in the paper.

\begin{corollary}
\label{useful}
Suppose $\mu$ is singular, while $\sigma<\theta$ are regular cardinals with $\sigma\leq\cf(\mu)<\theta$.  If
\begin{equation}
\pp_{\Gamma(\sigma)}(\mu)<\pp_{\Gamma(\theta,\sigma)}(\mu),
\end{equation}
then $\mu$ is a limit of singular cardinals $\eta$ such that
\begin{itemize}
\item $\sigma\leq\cf(\eta)<\theta<\eta<\mu$,
\item $\eta$ is eventually $\Gamma(\theta,\sigma)$-closed, and
\item $\pp_{\Gamma(\theta,\sigma)}(\mu)=\pp_{\Gamma(\theta,\sigma)}(\eta)$.
\end{itemize}
\end{corollary}
\begin{proof}
Our assumption implies that Option 2 of the Pseudopower Dichotomy holds by way of Theorem~\ref{theoremclosed}, and now the result follows immediately.
\end{proof}

\subsection{More on improving representations}

The proof of the Pseudopower Dichotomy provides a template for its applications.For example, suppose that a cardinal $\lambda$ is $\Gamma(\theta,\sigma)$-representable at $\mu$ and look at the Pseudopower Dichotomy. If Option~1 holds, then our work in the preceding section applies and we can get nice representations of $\lambda$ at $\mu$.  If, on the other hand, Option 2 holds sway, then we know instead that $\lambda$ has nice representations for many cardinals below~$\mu$, and we will be able to combine these representations (using the Continuity Property of pseudopowers from Proposition~\ref{continuity}) to produce a nice representation of~$\lambda$ at~$\mu$ itself.  The following definition will help us take advantage of these ideas.

\begin{definition}
\label{Xdef}
Suppose $\mu$ is singular, and $\sigma<\theta$ are regular with $\sigma\leq\cf(\mu)<\theta$.  We define $\mathcal{X}_{\Gamma(\theta,\sigma)}(\mu)$ to be the collection of cardinals~$\eta$ satisfying
\begin{itemize}
\item $\sigma\leq\cf(\eta)<\theta<\eta<\mu$,
\sk
\item $\eta$ is eventually $\Gamma(\theta,\sigma)$closed, and
\sk
\item $\PP_{\Gamma(\theta,\sigma)}(\mu)=\PP_{\Gamma(\theta,\sigma)}(\eta)\setminus\mu^+$.
\sk
\end{itemize}
\end{definition}

This set might very well be empty, but note that the Pseudopower Dichotomy can be reformulated as the statement ``either $\mathcal{X}_{\Gamma(\theta,\sigma)}(\mu)$ is unbounded in~$\mu$ or it is not''.  If $\mathcal{X}_{\Gamma(\theta,\sigma)}(\mu)$ is unbounded, then it will exert an influence over representations at~$\mu$ through continuity. For example, we have the following observation which yields the same conclusion as Theorem~\ref{theoremclosed}.

\begin{proposition}
\label{4.3}
Suppose $\{\eta\in \mathcal{X}_{\Gamma(\theta,\sigma)}(\mu): \cf(\eta)\leq\cf(\mu)\}$ is unbounded in~$\mu$.  Then
\begin{equation}
\PP_{\Gamma(\theta,\sigma)}(\mu)=\PP_{\Gamma(\sigma)}(\mu).
\end{equation}
\end{proposition}
\begin{proof}
Suppose~$\lambda$ is $\Gamma(\theta,\sigma)$-representable at~$\mu$. Since Option~2 holds, we know that $\lambda$ is also $\Gamma(\theta,\sigma)$-representable at each $\eta\in \mathcal{X}_{\Gamma(\theta,\sigma)}(\mu)$. Given such an $\eta$, we can apply Theorem~\ref{theoremclosed} to conclude $\lambda$ is also $\Gamma((\cf\eta)^+,\sigma))$-representable at~$\eta$.
By the assumption of the lemma, we have
\begin{equation}
\mu = \sup\{\eta<\mu:\sigma\leq\cf(\eta)<(\cf(\mu))^+\text{ and }\lambda\in\PP_{\Gamma((\cf\mu)^+,\sigma)}(\eta)\},
\end{equation}
and hence $\lambda$ is $\Gamma(\sigma)$-representable at $\mu$
by way of Proposition~\ref{continuity}.
\end{proof}

The next result builds on this idea, and shows how assumptions on the structure of $\mathcal{X}_{\Gamma(\theta,\sigma)}(\mu)$ let us improve representations of cardinals in $\PP_{\Gamma(\theta,\sigma)}(\mu)$.

\begin{theorem}
\label{reduction}
Suppose $\mu$ is singular, and $\sigma$,$\tau$, $\chi$, and $\theta$ are regular cardinals satisfying
\begin{equation}
\sigma<\tau\leq\cf(\mu)<\chi\leq\theta<\mu.
\end{equation}
If
\begin{equation}
\mu = \sup\{\eta\in\mathcal{X}_{\Gamma(\theta,\sigma)}(\mu): \sup\{\cov(\rho,\rho,\sigma,2)^+:\sigma\leq\rho<\tau\}\leq\cf(\eta)<\chi\},
\end{equation}
then
\begin{equation}
\PP_{\Gamma(\theta,\sigma)}(\mu)=\PP_{\Gamma(\chi,\tau)}(\mu),
\end{equation}
hence
\begin{equation}
\pp_{\Gamma(\theta,\sigma)}(\mu)=\pp_{\Gamma(\chi,\tau)}(\mu)
\end{equation}
\end{theorem}
\begin{proof}
Clearly $\PP_{\Gamma(\chi,\tau)}(\mu)$ is contained in $\PP_{\Gamma(\theta,\sigma)}(\mu)$, so we will show the reverse inclusion holds.   To do this, suppose $\eta$ satisfies the following conditions:

\begin{itemize}
\item $\lambda$ is $\Gamma(\theta,\sigma)$-representable at $\eta$,
\sk
\item $\eta$ is eventually $\Gamma(\theta,\sigma)$-closed, and
\sk

\item $\sigma\leq\rho<\tau \Longrightarrow \cov(\rho,\rho,\sigma,2)<\tau\leq\cf(\eta)$.

\sk
\end{itemize}
An application of Theorem~\ref{theoremgeneral} tells us that $\lambda$ is $\Gamma((\cf\eta)^+,\tau)$-representable at $\eta$, and since $\cf(\eta)<\chi$ this means $\lambda$ is $\Gamma(\chi,\tau)$-representable at $\eta$.
We have assumed that the set of such $\eta$ is unbounded in $\mu$, so once again Proposition~\ref{continuity} forces $\lambda$ to be $\Gamma(\chi,\tau)$-representable at~$\mu$ as well.
\end{proof}

\section{Applications of the Pseudopower Dichotomy}

\subsection{Computing $\PP_\theta(\mu)$} In this section, we put together pieces from our preceding work to obtain theorems in {\sf ZFC} based on the Pseudopower Dichotomy.  Our first result echoes Theorem~\ref{371}, as tells us that $\PP_\theta(\mu)$ can be computed from pseudopowers that involve the cofinality of $\mu$ in two different ways.

Since $\PP(\mu)$ is defined as $\PP_{\cf(\mu)}(\mu)$, it is clear that
\begin{equation}
\label{5.1a}
\PP(\mu)\cup \PP_{\Gamma(\theta^+,\cf(\mu))}\subseteq\PP_{\theta}(\mu).
\end{equation}
The first theorem in this section shows that the two sides of (\ref{5.1a}) are actually equal.

\begin{theorem}
\label{theorem7}
Suppose $\mu$ is singular, and $\cf(\theta)\leq\theta<\mu$.
Then
\begin{equation}
\label{eqnhuh}
\PP_\theta(\mu) = \PP(\mu)\cup \PP_{\Gamma(\theta,\cf(\mu))}(\mu).
\end{equation}
hence
\begin{equation}
\pp_\theta(\mu)=\pp(\mu)+\pp_{\Gamma(\theta^+,\cf(\mu))}(\mu).
\end{equation}
\end{theorem}

We make a couple of comments before presenting the proof. First, note that the above theorem shows us that if $\lambda$ has a representation at $\mu$ using a set of size~$\theta$, then $\lambda$
can be represented at~$\mu$ using either a set of cardinality~$\cf(\mu)$ (the minimum possible size) or a $\cf(\mu)$-complete ideal (the maximum possible). These options are not mutually exclusive (for example, both hold simultaneously for $\mu^+$), but the power of the theorem is in the statement that at least of
these two things {\bf\em must} occur.

Our second comment is to note that because $\PP(\mu)$ and $\PP_{\Gamma(\theta,\cf(\mu))}(\mu)$ are both intervals of regular cardinals, the equation (\ref{eqnhuh}) implies that $\PP_\theta(\mu)$ must in fact be equal to one of these two sets.

\begin{proof}[Proof of Theorem~\ref{theorem7}]
We may assume that $\mu$ has uncountable cofinality, as otherwise (\ref{eqnhuh}) holds automatically. We apply the pseudopower dichotomy to $\Gamma(\theta^+,\aleph_0)$.   If $\mu$ is eventually $\Gamma(\theta^+,\aleph_0)$-closed, then Corollary~\ref{3.6} gives us more than we need because any $\lambda$ in $\PP_\theta(\mu)$ can be represented using the bounded ideal on a set of cardinality~$\cf(\mu)$ cofinal in $\mu\cap\reg$.  It follows that in this situation, all three of the sets from
(\ref{eqnhuh}) are equal, and
\begin{equation}
\PP_{\Gamma(\theta^+,\aleph_0)}(\mu)=\PP_{\Gamma(\cf(\mu))}(\mu)\subseteq\PP(\mu)=\PP_{\Gamma(\theta^+,\cf(\mu))}(\mu),
\end{equation}
which is more than we require.

Thus, we may assume that Option~2 of the Pseudopower Dichotomy is in force and the corresponding set~$\mathcal{X}_{\Gamma(\theta^+,\aleph_0)}$ is unbounded in~$\mu$. Abbreviating this set as ``$\mathcal{X}$'', we split into two cases.

If the set of $\eta\in \mathcal{X}$ with $\cf(\eta)\leq\cf(\mu)$ is unbounded in~$\mu$, then Proposition~\ref{4.3} tells us
\begin{equation}
\PP_{\Gamma(\theta^+,\aleph_0)}(\mu)=\PP(\mu),
\end{equation}
and (\ref{eqnhuh}) is immediate.

If  $\{\eta\in\mathcal{X}: \cf(\eta)>\cf(\mu)\}$ is unbounded in $\mu$, then we apply Theorem~\ref{reduction} with $\chi = \theta^+$, $\upsilon=\cf(\mu)$, and $\sigma$ and $\tau$ both equal to~$\aleph_0$ to conclude
\begin{equation}
\PP_\theta(\mu)=\PP_{\Gamma(\theta^+,\cf(\mu))}(\mu),
\end{equation}
which establishes (\ref{eqnhuh}).  Note that Theorem~\ref{reduction} does apply here: since $\sigma=\aleph_0$, we know $\cov(\rho,\rho,\sigma, 2) = \rho$ for any $\rho<\cf(\mu)$.

Clearly at least one of these cases must happen, and therefore (\ref{eqnhuh}) holds under Option~2 of the Pseudopower Dichotomy, finishing the proof.
\end{proof}

Reformulating the preceding result in less technical language, we have shows that given a singular cardinal $\mu$, if a cardinal $\lambda$ is representable at $\mu$ using a set of cardinality at most $\theta<\mu$, then either $\lambda$ is representable at $\mu$ using a set of cardinality $\cf(\mu)$, or $\lambda$ can be represented at $\mu$ using a~$\cf(\mu)$-complete ideal on a set of cardinality at most~$\theta$.

\subsection{Computing $\PP_{\Gamma(\theta,\sigma)}(\mu)$} The next theorem examines the more general situation where $\sigma$ may be uncountable. The conclusion is weaker than that of Theorem~\ref{theorem7}, but the fact that $\sigma$ is uncountable will allow us  to transfer the result into a corresponding statement about covering numbers.

\begin{theorem}
\label{theorem8}
Suppose $\mu$ is singular, and let $\sigma$ and $\theta$ be regular cardinals with $\sigma<\cf(\mu)<\theta<\mu$. Then
\begin{equation}
\label{eqnbleh}
\PP_{\Gamma(\theta,\sigma)}(\mu)= \PP_{\Gamma(\sigma)}(\mu)\cup \PP_{\Gamma(\theta,\tau)}(\mu)
\end{equation}
for any regular $\tau\in (\sigma,\cf(\mu)]$ such that
\begin{equation}
\label{5.9}
\sigma\leq\rho<\tau\Longrightarrow \cov(\rho,\rho,\sigma, 2)<\cf(\mu).
\end{equation}
In particular, for such $\tau$ we have
\begin{equation}
\pp_{\Gamma(\theta,\sigma)}(\mu)=\pp_{\Gamma(\sigma)}(\mu) + \pp_{\Gamma(\theta,\tau)}(\mu).
\end{equation}
\end{theorem}

\begin{proof}
The proof mirrors that of Theorem~\ref{theorem7}.  If $\mu$ is eventually $\Gamma(\theta,\sigma)$-closed, then all three sets mentioned in~(\ref{eqnbleh}) are equal, and the result follows.  Thus, we may assume that Option 2 of the Pseudopower Dichotomy is in force, and the set $\mathcal{X}=\mathcal{X}_{\Gamma(\theta,\sigma)}(\mu)$ from Definition~\ref{Xdef} is unbounded in~$\mu$.  We break into cases just as in Theorem~\ref{theorem7}. If the set of $\eta$ in $\mathcal{X}$ with $\cf(\eta)\leq\cf(\mu)$ is unbounded, then
\begin{equation}
\PP_{\Gamma(\theta,\sigma)}(\mu)=\PP_{\Gamma(\sigma)}(\mu)
\end{equation}
by way of Proposition~\ref{4.3}.

If the set of $\eta\in\mathcal{X}$ with $\cf(\mu)<\cf(\eta)$ is unbounded in~$\mu$, then by Theorem~\ref{reduction} we will have
\begin{equation}
\PP_{\Gamma(\theta,\sigma)}(\mu)=\PP_{\Gamma(\theta,\tau)}(\mu).
\end{equation}
Since at least one of these two things must occur under Option 2, the result follows.
\end{proof}

Note that condition (\ref{5.9}) holds if $\sigma\leq\tau<\sigma^{+\omega}$, so we are able to conclude the following:

\begin{corollary}
\label{corcor}
Suppose $\mu$ is singular, and $\sigma$ and $\theta$ are regular cardinals such that $$\sigma^{+n}\leq\cf(\mu)<\theta<\mu.$$
Then
\begin{equation}
\PP_{\Gamma(\theta,\sigma)}(\mu)=\PP_{\Gamma(\sigma)}(\mu)\cup\PP_{\Gamma(\theta,\sigma^{+n})}(\mu),
\end{equation}
and thus
\begin{equation}
\pp_{\Gamma(\theta,\sigma)}(\mu) = \pp_{\Gamma(\sigma)}(\mu) + \pp_{\Gamma(\theta,\sigma^{+n})}(\mu).
\end{equation}
In particular, if $\sigma<\cf(\mu)$, then
\begin{equation}
\label{5.2}
\pp_{\Gamma(\theta,\sigma)}(\mu)=\pp_{\Gamma(\sigma)}(\mu) + \pp_{\Gamma(\theta,\sigma^+)}(\mu).
\end{equation}
\end{corollary}

This is similar in spirit to results we saw earlier: if we know a cardinal is $\Gamma(\theta,\sigma)$-representable at~$\mu$ with $\sigma<\cf(\mu)$, then either it is representable using a $\sigma$-complete ideal on a set of cardinality~$\cf(\mu)$ (so the size is as small as
possible), or it is representable on a set of cardinality less than $\theta$, but using a $\sigma^+$-complete ideal (so we are able to find a representation with greater completeness). The equation (\ref{5.2}) will be important in the next section, in which we use work of Gitik to show that both terms appearing on the right-hand side of the equality are needed.

\subsection{A theorem of Gitik and Shelah} For our final application, we show that the pseudopower dichotomy underlies a result of Gitik and Shelah (Theorem~1.5 of~\cite{412}) on cardinal arithmetic. They phrase their result in terms of covering numbers of the form $\cov(\mu,\kappa,\theta,\aleph_1)$, but their proof is basically as given below.
\begin{theorem}[Gitik-Shelah~\cite{412}]
\label{gitikshelah}
Suppose $\sigma\leq\cf(\mu)<\mu$.  Then
\begin{equation}
\{\pp_{\Gamma(\theta,\sigma)}(\mu):\cf(\mu)<\theta=\cf(\theta)<\mu\}\text{ is finite.}
\end{equation}
\end{theorem}

As noted by Gitik and Shelah, this result can be viewed as ``non-GCH analog" of a theorem of Hajnal and Shelah that $\{\mu^\theta: 2^\theta<\mu\}$ is finite.\footnote{One reference for this result given in~\cite{412} is incorrect: it appears as a Exercise 5 in Section 11 of~\cite{HH}.  Shelah noted it independently on pages 164 and 165 of~\cite{232}.}

\begin{proof}
Suppose this fails for $\mu$, and choose an increasing sequence $\langle\theta_n:n<\omega\rangle$ such that for each $n<\omega$,
\begin{equation}
\cf(\mu)<\theta_n=\cf(\theta_n)<\mu,
\end{equation}
and
\begin{equation}
\pp_{\Gamma(\theta_n,\sigma)}(\mu)<\pp_{\Gamma(\theta_{n+1},\sigma)}(\mu).
\end{equation}
Since
\begin{equation}
\pp_{\Gamma(\sigma)}(\mu)\leq\pp_{\Gamma(\theta,\sigma)}(\mu)
\end{equation}
for any regular $\theta$ in the interval $(\cf(\mu),\mu)$, we can assume that
\begin{equation}
\pp_{\Gamma(\sigma)}(\mu)<\pp_{\Gamma(\theta_n,\sigma)}(\mu)
\end{equation}
as well.  By Corollary~\ref{useful}, we know for each $n<\omega$ that $\mu$ is a limit of eventually $\Gamma(\theta_n,\sigma)$-closed cardinals $\eta$  with
\begin{equation}
\pp_{\Gamma(\theta_n, \sigma)}(\eta)=\pp_{\Gamma(\theta_n, \sigma)}(\mu).
\end{equation}
For each $n$, let $\eta_n$ be the least such cardinal.  The sequence $\langle\eta_n:n<\omega\rangle$ is non-increasing, hence eventually constant, say with value $\eta$.  This particular $\eta$ will be eventually $\Gamma(\theta_n,\sigma)$-closed for each $n<\omega$, and hence
\begin{equation}
\pp_{\Gamma(\theta_n, \sigma)}(\eta) = \pp_{\Gamma(\sigma)}(\eta)
\end{equation}
by Theorem~\ref{theoremclosed}.  But then for any $n$, we have
\begin{equation}
\pp_{\Gamma(\theta_n,\sigma)}(\mu) = \pp_{\Gamma(\theta_n,\sigma)}(\eta)=\pp_{\Gamma(\sigma)}(\eta)=\pp_{\Gamma(\theta_{n+1},\sigma)}(\eta)=\pp_{\Gamma(\theta_n,\sigma)}(\mu),
\end{equation}
which contradicts our choice of the sequence $\langle\theta_n:n<\omega\rangle$.
\end{proof}

This result will reappear in the final section of our paper, when we formulate some natural open questions.

\section{Independence Results}

\subsection{Gitik's theorem} This short section will focus on obtaining independence results complementary to the theorems we established in the previous section. Specifically, we will examine the formula
\begin{equation}
\label{6.1}
\pp_{\Gamma(\theta,\sigma)}(\mu)=\pp_{\Gamma(\sigma)}(\mu)+\pp_{\Gamma({\theta,\sigma^+})}(\mu)
\end{equation}
established in Corollary~\ref{corcor}. We rely almost completely on recently published work of Moti Gitik~\cite{gitik}. Since we do not have the expertise to discuss his proof in detail, our approach will be to quote his results liberally. We start with his main theorem:

\begin{theorem}[Theorem 1.3 of~\cite{gitik}]
\label{gitikthm} Assume GCH.  Let $\eta$ be an ordinal and $\delta$ be a regular cardinal.  Let $\langle
\kappa_\alpha:\alpha<\eta\rangle$ be an increasing sequence of strong cardinals, and let $\lambda$ be a cardinal greater than the
supremum of $\{\kappa_\alpha:\alpha<\eta\}$.  Then there is a cardinal preserving extension in which, for every $\alpha<\eta$,
\begin{enumerate}
\item $\cf(\kappa_\alpha)=\delta$, and \sk
\item $\pp(\kappa_\alpha)\geq\lambda$.
\end{enumerate}
\end{theorem}

The final section of his paper (Section 8) is the part most relevant for us, as he examines the cardinal arithmetic structure of
his model in some detail. We will follow his notation, and point out what need.

\subsection{Computations} We assume that $\delta$ and $\eta$ are regular cardinals with $\aleph_2\leq\delta$ and $\delta^+<\eta$, and we work in the
generic extension $V[G]$ from Theorem~\ref{gitikthm}. In $V[G]$, the cardinals $\kappa_\alpha$ for $\alpha<\eta$ will all have
cofinality $\delta$, and each satisfies $\pp(\kappa_\alpha)\geq\lambda$.   We need a little more: as Gitik notes prior to his
Proposition~8.11, in fact we have the stronger result that
\begin{equation}
\pp_{\Gamma(\delta)}(\kappa_\alpha)\geq\lambda
\end{equation}
as all of the ideals involved in the computation are $\delta$-complete.

For each $\alpha<\eta$, we let
\begin{equation}
\bar{\kappa}_\alpha = \sup\{\kappa_\beta:\beta<\alpha\}.
\end{equation}
In $V[G]$, if $\alpha<\eta$ is a limit ordinal, then $\bar{\kappa}_\alpha$ is singular with cofinality less than $\eta$. Our
choices of $\delta$ and $\eta$ guarantee that there are limit ordinals $\alpha<\eta$ with $\cf(\alpha)$ greater than $\delta$,
and others of cofinality less than~$\delta$ but greater than $\omega_1$.  The corresponding $\bar{\kappa}_\alpha$ for these two
sorts of $\alpha$ are the places of interest to us, and we analyze each situation separately.  Again, relying on Gitik's work we
have:

\begin{proposition}
Let $\alpha$ be a limit ordinal less than $\eta$, and let $V[G]$ be as in Theorem~\ref{gitikthm}.  Then the model $V[G]$
satisfies:
\begin{enumerate}
\item If $\cf(\alpha)<\delta$ then
\begin{equation}
\label{6.3}
    \pp(\bar{\kappa}_\alpha)<\pp_{\Gamma(\delta^+,\cf(\alpha))}(\bar{\kappa}_\alpha).
    \end{equation}

\item If $\delta<\cf(\alpha)$ then
\begin{equation}
\label{6.4}
\pp_{\Gamma(\delta^+)}(\bar{\kappa}_\alpha)<\pp_{\Gamma(\delta)}(\bar{\kappa}_\alpha).
\end{equation}
\end{enumerate}
\end{proposition}
\begin{proof}

For (1), suppose $\alpha<\eta$ is a limit ordinal of cofinality less than~$\delta$.  We know that $\{\kappa_\beta:\beta<\alpha\}$
is cofinal in $\bar{\kappa}_{\alpha}$, and each $\kappa_\beta$ is singular of cofinality $\delta$ with
\begin{equation}
\pp_{\Gamma(\delta)}(\kappa_\beta)\geq\lambda,
\end{equation}
hence
\begin{equation}
\pp_{\Gamma(\delta^+,\cf(\alpha))}(\kappa_\beta)\geq\lambda.
\end{equation}
An application of Continuity tells us
\begin{equation}
\pp_{\Gamma(\delta^+,\cf(\alpha))}(\bar{\kappa}_\alpha)\geq\lambda
\end{equation}
as well.  On the other hand, Gitik's Proposition~8.6 tells us
\begin{equation}
\pp(\bar{\kappa}_\alpha)=\kappa_\alpha,
\end{equation}
and so
\begin{equation}
\pp(\bar{\kappa}_\alpha)=\kappa_\alpha<\lambda\leq\pp_{\Gamma(\delta^+,\cf(\alpha))}(\bar{\kappa}_\alpha)
\end{equation}
as required.

For (2), suppose $\alpha<\eta$ is a limit ordinal of cofinality greater than~$\delta$. Again, the set
$\{\kappa_\beta:\beta<\alpha\}$ is unbounded in $\bar{\kappa}_\alpha$, and so an application of continuity tells us
\begin{equation}
\pp_{\Gamma(\delta)}(\bar{\kappa}_\alpha)\geq\lambda.
\end{equation}
On the other hand, Gitik's Proposition~8.11 tells us
\begin{equation}
\pp_{\Gamma(\delta^+)}(\bar{\kappa}_\alpha)\leq\kappa_\alpha,
\end{equation}
and~(\ref{6.4}) follows immediately.
\end{proof}

\subsection{Conclusions} Let us now return to the equation~(\ref{6.1}), and work with Gitik's extension $V[G]$ in the situation where $\delta$ is at least
$\omega_2$ and $\eta$ is greater than $\delta^+$ (these restrictions are simply to make sure we have enough room to manipulate
parameters of interest to us).

One the one hand, if we define $\mu = \bar{\kappa}_{\omega_1}$, $\theta=\delta^+$, and $\sigma = \aleph_0$, then by
equation~(\ref{6.3}), we have
\begin{equation}
\pp_{\Gamma(\sigma)}(\mu)\leq \pp(\mu) =\pp(\bar{\kappa}_\alpha)<\pp_{\Gamma(\delta^+,\cf(\alpha))}(\bar{\kappa}_\alpha) = \pp_{\Gamma(\theta,\sigma^+)}(\mu).
\end{equation}
and so for this choice of parameters we have
\begin{equation}
\label{6.14}
\pp_{\Gamma(\theta,\sigma)}(\mu)=\pp_{\Gamma(\theta,\sigma^+)}(\mu)>\pp_{\Gamma(\sigma)}(\mu).
\end{equation}

On the other hand, if we let $\mu = \bar{\kappa}_{\delta^+}$, $\theta=\delta^+$, and $\sigma = \delta$, then from
equation~(\ref{6.4}) we conclude
\begin{equation}
\pp_{\Gamma(\sigma^+)}(\mu)<\pp_{\Gamma(\sigma)}(\mu),
\end{equation}
and so (remembering that $\mu$ is of cofinality~$\theta$) we have
\begin{equation}
\label{6.16}
\pp_{\Gamma(\theta,\sigma)}(\mu) =
\pp_{\Gamma(\sigma)}(\mu)>\pp_{\Gamma(\theta,\sigma^+)}(\mu).
\end{equation}

Taken together, (\ref{6.14}) and (\ref{6.16}) show us that both summands in equation (\ref{6.1}) are important if we want a
theorem that holds in {\sf ZFC}.  We have not pushed the analysis of the cardinal arithmetic in Gitik's model beyond what was
presented above; it may be that there are similar examples for many other values of $\theta$ and $\sigma$.

\section{On covering numbers}

\subsection{Back to the beginning} We now look at what Theorem~\ref{theorem8} and its relatives tell us about covering numbers at singular cardinals.  Once again, we assume that $\sigma$ and $\theta$ are infinite regular cardinals, while $\mu$ is a singular cardinal satisfying
\begin{equation}
\sigma\leq\cf(\mu)<\theta<\mu.
\end{equation}

Recall that the covering number $\cov(\mu,\mu,\theta,\sigma)$ is defined to be the minimum cardinality of a subset $\mathcal{P}$ of $[\mu]^{<\mu}$ that $\sigma$-covers $[\mu]^{<\theta}$, that is, such that for every $X\in [\mu]^{<\theta}$ there is a subset $\mathcal{Y}$ of $\mathcal{P}$ of cardinality less than $\sigma$ with
\begin{equation}
X\subseteq\bigcup\mathcal{Y}.
\end{equation}
Recall as well that the {\sf cov vs. pp Theorem} of Shelah (Theorem~\ref{covppthm} mentioned in our introduction) tells us that if $\sigma$ is uncountable, then
\begin{equation}
\cov(\mu,\mu,\theta,\sigma)=\pp_{\Gamma(\theta,\sigma)}(\mu).
\end{equation}

Given this, the following theorem is a straightforward  translation of Theorem~\ref{theorem8} into the language of covering numbers.

\begin{theorem}
\label{theorem9}
Suppose $\mu$ is a singular cardinal, and let $\sigma$ and $\theta$ be regular cardinals such that
\begin{equation}
\aleph_0<\sigma<\cf(\mu)<\theta.
\end{equation}
Then
\begin{equation}
\label{odd}
\cov(\mu,\mu,\theta,\sigma)=\cov(\mu,\mu,(\cf\mu)^+,\sigma)+\cov(\mu,\mu,\theta,\sigma^+).
\end{equation}
\end{theorem}
\begin{proof}
 By (\ref{5.2}), we know
\begin{equation}
\pp_{\Gamma(\theta,\sigma)}(\mu)=\pp_{\Gamma(\sigma)}(\mu)+\pp_{\Gamma(\theta,\sigma^+)}(\mu).
\end{equation}
Because $\sigma$ is uncountable, we invoke Theorem~\ref{covppthm} and convert the pseudopowers into covering numbers, yielding
\begin{equation}
\label{zfc}
\cov(\mu,\mu,\theta,\sigma)=\cov(\mu,\mu,(\cf\mu)^+,\sigma)+\cov(\mu,\mu,\theta,\sigma^+)
\end{equation}
as required.
\end{proof}

Similarly, based on Corollary~\ref{corcor}, we have the following:

\begin{corollary}
\label{cor7.1} Suppose $\mu$ is a singular cardinal, and $\sigma<\theta$ are regular cardinals such that
\begin{equation}
\aleph_0<\sigma\leq\cf(\mu)<\min(\sigma^{+\omega},\theta)\leq\theta<\mu,
\end{equation}
Then
\begin{equation}
\cov(\mu,\mu,\theta,\sigma)=\cov(\mu,\mu,(\cf\mu)^+,\sigma)+\cov(\mu,\mu,\theta,\cf(\mu)).
\end{equation}
\end{corollary}

In particular, if $\mu$ is singular of cofinality $\aleph_6$, and we set $\theta=\aleph_9$ and $\sigma=\aleph_2$, then Corollary~\ref{cor7.1} tells us
\begin{equation}
\label{bizarre}
\cov(\mu,\mu,\aleph_9,\aleph_2) = \cov(\mu,\mu,\aleph_7, \aleph_2) + \cov(\mu,\mu,\aleph_9,\aleph_6).
\end{equation}
and we finally arrive at (\ref{amusement}) from the introduction.  What important for this choice of parameters is
\begin{equation}
\sigma\leq\cf(\mu)<\sigma^{+\omega},
\end{equation}

\subsection{What does it mean?}  The above formula~(\ref{bizarre}) hides a trichotomy about the ways in which certain elementary submodels interact with subsets of~$\mu$. To see why, let $\chi$ be a sufficiently large regular cardinal, and suppose
$M$ is an elementary submodel of $H(\chi)$ containing $\mu$ and such that
\begin{equation}
|M|+1\subseteq M.
\end{equation}
If we let $\mathcal{P}$ be $M\cap [\mu]^{<\mu}$, then EXACTLY one of the following three things MUST occur:
\begin{enumerate}
\item Some $X\in [\mu]^{\aleph_6}$ cannot be covered by a union of $\aleph_1$ sets from~$\mathcal{P}$.
\sk
\item Every $X\in [\mu]^{\aleph_8}$ can be covered by a union of $\aleph_1$ sets from $\mathcal{P}$.
\sk
\item Every $X\in [\mu]^{\aleph_6}$ is covered by a union of $\aleph_1$ sets from $\mathcal{P}$, but some $Y\in [\mu]^{\aleph_8}$ cannot be covered by a union of $\aleph_5$ sets from $\mathcal{P}$.
\end{enumerate}

The above is quite easy: if we let $\kappa$ be the cardinality of $M$, then (\ref{bizarre}) tells us that exactly one of the following must be true:

\begin{enumerate}
\item[$(1)'$] $\kappa<\cov(\mu,\mu,\aleph_7, \aleph_2)$, or
\sk
\item[$(2)'$] $\cov(\mu,\mu,\aleph_9, \aleph_2)\leq\kappa$, or
\sk
\item[$(3)'$] $\cov(\mu,\mu,\aleph_7,\aleph_2)\leq\kappa<\cov(\mu,\mu,\aleph_9, \aleph_6)$
\end{enumerate}
The rest follows easily.

More generally, we can consider this view of our conclusion~(\ref{odd}). Given a model $M$ as above,  if every member of $[\mu]^{\cf(\mu)}$ can be covered by a union of {\bf\em fewer than} $\sigma$ sets from $M\cap [\mu]^{<\mu}$, and every member of $[\mu]^{<\theta}$ can be covered by a union of {\bf\em at most} $\sigma$ sets from $M\cap [\mu]^{<\mu}$, then in fact every member of $[\mu]^{<\theta}$ can be covered by fewer than $\sigma$ sets from $M\cap [\mu]^{<\mu}$.

\subsection{The role of pseudopowers}  Looking back at the discussion ending the previous section, the conclusions speaks about elementary submodels and the combinatorics of $[\mu]^{<\mu}$ without any reference at all to the pseudopowers that were used in the proofs.  The fact that we are forced to use pseudopowers and then rely on the cov vs. pp theorem means that any conclusions on covering numbers are currently limited to situations where the last argument $\sigma$ is uncountable. Thus, eliminating the reliance on pseudopowers might unlock a more general theorem:
\begin{question}
Can results like Theorem~\ref{theorem9} be proved directly without the use of pseudopowers?
\end{question}
It is possible that an argument like Shelah's original proof of the cov vs. pp theorem (pp. 87-93 of~\cite{cardarith}) may work under these circumstances.

\subsection{Gitik-Shelah revisited}

Look back at Theorem~\ref{gitikshelah}.  This tells us that for a singular cardinal $\mu$, if we fix a regular cardinal $\sigma\leq\cf(\mu)$ then as $\theta$ ranges over the interval $(\cf(\mu),\mu)\cap\reg$ there are only finitely many distinct pseudopowers $\pp_{\Gamma(\theta,\sigma)}(\mu)$ achieved.    If we instead fix $\theta\in (\cf(\mu),\mu)\cap\reg$ and let $\sigma$ range over the regular cardinals in $[\aleph_0,\cf(\mu)]$, then we achieve only finitely many distinct values for different reasons: as $\sigma$ increases, the corresponding sequence of pseudopowers is non-increasing.  Thus, it is natural to ask if there is something deeper going on.

\begin{question}
Suppose $\mu$ is a singular cardinal.
\begin{equation}
\text{ Is }\{\pp_{\Gamma(\theta,\sigma)}(\mu):\sigma=\cf(\sigma)\leq\cf(\mu)<\theta=\cf(\theta)<\mu\}\text{ finite?}
\end{equation}
\end{question}
Note that this is true for $\mu$ satisfying $\cf(\mu)<\aleph_\omega$, so the first interesting case occurs when $\mu$ is singular of cofinality $\aleph_{\omega+1}$.  For a more specific question that may shed light, suppose $\mu$ is singular of cofinality $\aleph_{\omega+1}$, and we have a cardinal $\lambda$ that is $\Gamma(\aleph_n)$-representable at $\mu$ for all $n<\omega$.  Is $\lambda$ also $\Gamma(\aleph_{\omega+1})$-representable at $\mu$?

\subsection{A conjecture}

Finally, we ask if the behavior we observed for singular cardinals with cofinality less than $\aleph_\omega$ holds in general.

\begin{question}
Suppose $\mu$ is singular, and $\sigma<\theta$ are regular cardinals with $\sigma\leq\cf(\mu)<\theta<\mu$.
\begin{equation}
\text{ Is }\PP_{\Gamma(\theta,\sigma)}(\mu) = \PP_{\Gamma((\cf\mu)^+,\sigma)}(\mu)\cup\PP_{\Gamma(\theta,\cf(\mu))}(\mu)\text{?}
\end{equation}
\end{question}

Note that a positive answer to this question also implies a positive answer to Question 2, as both
\begin{equation}
\left\{\pp_{\Gamma((\cf(\mu))^+,\sigma)}:\sigma\in [\aleph_0,\cf(\mu)]\cap\reg\right\}
\end{equation}
and
\begin{equation}
\left\{\pp_{\Gamma(\theta,\cf(\mu))}(\mu):\theta\in(\cf(\mu),\mu)\cap\reg\right\}
\end{equation}
are finite.  We conjecture that the answer to Question~3 is positive.

\bibliographystyle{plain}

\end{document}